\documentclass[10pt]{amsart}

\usepackage{color}
  \usepackage{ifpdf}
  \ifpdf
    \usepackage[pdftex,pagebackref,colorlinks,linkcolor=blue,citecolor=blue,urlcolor=blue,a4paper,hypertexnames=true,plainpages=false]{hyperref}
    \hypersetup{pdfauthor={Clara L\"oh, Roman Sauer},
                pdftitle ={Simplicial volume of Hilbert modular varieties},
                }
    \usepackage[alphabetic]{amsrefs}
    \usepackage[all]{xy}
    \SelectTips{cm}{}
  \else
    \usepackage[pagebackref,colorlinks,linkcolor=blue,citecolor=blue,urlcolor=blue,a4paper,hypertexnames=true,plainpages=false]{hyperref}
    \usepackage[alphabetic]{amsrefs}
    \usepackage[all, dvips]{xy}
    \SelectTips{cm}{}
  \fi

\newcommand{\bbH}{{\mathbb H}}
\newcommand{\bbN}{{\mathbb N}}
\newcommand{\bbQ}{{\mathbb Q}} 
\newcommand{\bbR}{{\mathbb R}} 
\newcommand{\bbZ}{{\mathbb Z}} 
\newcommand{\bbC}{{\mathbb C}} 

\newcommand{\bfG}{{\mathbf G}}
\newcommand{\bfP}{{\mathbf P}}

\newcommand{\bfN}{{\mathbf N}}

\newcommand{\frakn}{\mathfrak{n}}
\newcommand{\fraka}{\mathfrak{a}}
\newcommand{\frakg}{\mathfrak{g}}

\newcommand{\calo}{\mathcal{O}}

\let\phi\varphi
\let\epsilon\varepsilon
\let\rho\varrho
\let\theta\vartheta

\newcommand{\abs}[1]{{\left\lvert #1\right\rvert}}

\newcommand{\ad}{\operatorname{ad}}

\newcommand{\bs}{\backslash}

\newcommand{\defq}{\mathrel{\mathop:}=}

\newcommand{\id}{\operatorname{id}}

\newcommand{\isom}{\operatorname{Isom}}

\newcommand{\lfvol}[1]{\norm{#1}^{\mathrm{lf}}}
\newcommand{\lipschitz}{\operatorname{Lip}}
\newcommand{\lipvol}[1]{\norm{#1}_{\mathrm{Lip}}}

\newcommand{\minvol}{\operatorname{minvol}}
\newcommand{\norm}[1]{{\left\lVert #1\right\rVert}}

\newcommand{\vol}{\operatorname{vol}}
\newcommand{\noqed}{\renewcommand{\qedsymbol}{}}
\newcommand{\ucov}[1]{\ensuremath{\widetilde{#1}}}

\newcommand{\Clf}{C^{\mathrm{lf}}}           

\newcommand{\sauer}[1]%
{\marginpar{\begin{minipage}[c]{\marginparwidth}%
                                  \scriptsize\emph{\color{red}Roman: #1}%
                                  \end{minipage}}%
}

\newcommand{\loeh}[1]%
{\marginpar{\begin{minipage}[c]{\marginparwidth}%
                                  \scriptsize\emph{\color{red}Clara: #1}%
                                  \end{minipage}}%
}

\newtheorem{theorem}{Theorem}[section]
\newtheorem{lemma}[theorem]{Lemma}

\newtheorem{corollary}[theorem]{Corollary}
\theoremstyle{definition}
\newtheorem{definition}[theorem]{Definition}

\newtheorem{setup}[theorem]{Setup}
\newtheorem{strategy}[theorem]{Strategy}
\newtheorem*{acknow}{Acknowledgments}
\numberwithin{equation}{section}

\title{Simplicial volume of Hilbert modular varieties}

\author{Clara L\"oh}
\address{Graduiertenkolleg ``Analytische Topologie und
  Metageometrie,'' Westf\"alische Wil\-helms-Universit\"at M\"unster,
  M\"unster, Germany}
\email{clara.loeh@uni-muenster.de}
\urladdr{wwwmath.uni-muenster.de/u/clara.loeh}

\author{Roman Sauer}
\address{University of Chicago, Chicago, USA}
\email{romansauer@member.ams.org}
\urladdr{www.romansauer.de}

\date{\today}
\subjclass[2000]{Primary 53C23; Secondary 53C35}

\begin{document}

\maketitle

\begin{abstract}
  The simplicial volume introduced by Gromov provides a topologically
  accessible lower bound for the minimal volume. Lafont and Schmidt
  proved that the simplicial volume of closed, locally symmetric
  spaces of non-compact type is positive. In this paper, we extend 
  this positivity result to certain non-compact locally
  symmetric spaces of finite volume, including Hilbert
  modular varieties. The key idea is to reduce the problem to the
  compact case by first relating the simplicial volume of these
  manifolds to the Lipschitz simplicial volume and then taking
  advantage of a proportionality principle for the Lipschitz
  simplicial volume.  Moreover, using computations of Bucher-Karlsson
  for the simplicial volume of products of closed surfaces, we obtain
  the exact value of the simplicial volume of Hilbert modular
  surfaces.
\end{abstract}


\section{Introduction and Statement of results}\label{sec:intro}

\subsection{Simplicial volume}\label{subsec:definition of simplicial volume}
The simplicial volume of a manifold was introduced by
Gromov~\cite{gromov}: The simplicial volume of an $n$-dimensional,
compact, oriented manifold $M$ 
(possibly with non-empty boundary~$\partial M$) is defined by
\[\norm{M,\partial M}=
\inf\bigl\{\abs{c}_1;~\text{$c\in C_n(M)$ relative fundamental cycle of
  $(M,\partial M)$}\bigr\}.
\]
Here $\abs{c}_1$ denotes the $\ell^1$-norm of a chain in the singular
chain complex $C_n(M)$ with real coefficients  
with respect to the basis of all singular simplices. 
Similarly, one defines the simplicial volume $\lfvol{M}$ for an open, 
oriented   
$n$-manifold $M$ by fundamental cycles in the locally finite 
chain complex $\Clf_n(M)$ with $\bbR$-coefficients: 
\[\lfvol{M}=
\inf\bigl\{\abs{c}_1;~\text{$c\in\Clf_n(M)$ fundamental cycle of
  $M$}\bigr\}.
\]
In the sequel we suppress the superscript~``lf'' in the notation and
just write $\norm{M}$.  Because both simplicial volumes are
multiplicative with respect to finite coverings, it makes sense to
define the simplicial volume of a non-orientable manifold~$M$
as~$\|\overline{M}\|/2$ where $\overline{M}$ is the 
orientation double covering of $M$. Notice that the simplicial volume
is invariant under proper homotopy equivalences.

\subsection{The volume estimate}\label{subsec:minimal volume}
One of the main motivations to introduce the simplicial volume is to
establish a lower bound for the minimal volume by a homotopy invariant
(proper homotopy invariant in the non-compact case).  Gromov defined
the minimal volume~$\minvol(M)$ of a smooth manifold~$M$ as the
infimum of volumes $\vol(M,g)$ over all complete Riemannian
metrics~$g$ on~$M$ with sectional curvature pinched between~$-1$
and~$1$.  The following fundamental inequality
holds~\cite{gromov}*{p.~12, p.~73} for all smooth $n$-manifolds~$M$:
\begin{equation*}
\norm{M}\le (n-1)^nn! \cdot \minvol(M).
\end{equation*}
A basic question is whether $\minvol(M)>0$, and thus whether $\norm{M}>0$. 
For example, for  
negatively curved locally symmetric spaces of finite volume one knows that 
the minimal volume is achieved by the locally 
symmetric metric~\cites{bcg, juan, storm}. Gromov first proved 
that the minimal volume of closed locally symmetric spaces 
is positive. Connell and Farb gave a much more detailed 
proof of this result in~\cite{chris+benson}. They also proved 
that locally symmetric spaces of finite volume that do not 
have a local $\bbH^2$- or $\mathrm{SL}(3,\bbR)/\mathrm{SO}(3,\bbR)$-factor 
(thereby excluding Hilbert modular varieties) have a positive 
minimal volume within the Lipschitz class of the locally symmetric 
metric.\footnote{Benson Farb informed us that positivity of the minimal 
volume in the finite volume case is erroneously 
stated in the cited 
reference without this Lipschitz constraint.}

\smallskip

Thurston 
proved that if $M$ has a complete Riemannian metric of finite volume 
with sectional curvature between $-k$ and $-1$ then 
$\norm{M}>0$~\cite{gromov}*{Section~0.2}. 
Moreover, if $M$ is a compact locally 
symmetric space of non-compact type, then $\norm{M}>0$ by a 
result of Lafont and Schmidt~\cite{ben-jean}.

\subsection{Lipschitz simplicial volume}\label{subsec:computations}
The aim of this paper is to give explicit 
computations and estimates of the simplicial volume of certain locally
symmetric spaces with $\bbR$-rank at least $2$, namely, of Hilbert
modular varieties.  

However, the simplicial volume of non-compact manifolds is much less well 
behaved than the one for compact manifolds. For instance, both 
the proportionality principle~\cite{gromov}*{Section~0.4} 
and the product formula~\cite{gromov}*{Section~0.2} fail in 
general. Even worse, the simplicial volume of $\Gamma\bs X$,  
where $X$ is a symmetric space and $\Gamma$ an arithmetic lattice 
of $\bbQ$-rank at least~$2$, vanishes~\cite{clara-roman}. 

The following metric variant 
of the simplicial volume, which was first considered 
by Gromov~\cite{gromov} and is studied in detail 
in our paper~\cite{clara-roman}, 
shows a much more decent behaviour. 

\begin{definition}\label{def:Lipschitz simplicial volume}
Let $M$ be an oriented, $n$-dimensional Riemannian manifold 
without boundary. For a locally 
finite chain $c=\sum_{i\ge 0} a_{\sigma_i}\sigma_i$ define 
$\lipschitz(c)\in [0,\infty]$ as the supremum of the 
Lipschitz constants of the singular simplices 
$\sigma_i:\Delta^n\rightarrow M$ where $\Delta^n$ carries 
the standard Euclidean metric. Then the  
\emph{Lipschitz simplicial volume}~$\lipvol{M}$ of~$M$ is defined as
\[\lipvol{M}=
\inf\bigl\{\abs{c}_1;~\text{$c\in\Clf_n(M)$ 
fundamental cycle of $M$ with $\lipschitz(c)<\infty$}\bigr\}.
\]
\end{definition}

Note that $\lipvol{M}=\norm{M}$ for all closed Riemannian
manifolds~$M$ because singular homology and smooth singular homology
are isometrically isomorphic with respect to the
$\ell^1$-semi-norm~\cite{loehmh}*{Proposition~5.3}.

The main advantage of the Lipschitz simplicial volume over the
ordinary one is the presence of the following proportionality
principles.   

\begin{theorem}[Gromov's proportionality principle, open -- closed]
\label{thm:gromov proportionality principle}  
Let $M$ be a closed Riemannian manifold, and let $N$ be a 
complete Riemannian manifold of finite volume. 
Assume the universal covers of $M$ and $N$ are isometric. 
Then 
\[\frac{\lipvol{M}}{\vol(M)}=\frac{\lipvol{N}}{\vol(N)}.\]
\end{theorem}

\begin{proof}
Gromov defines different notions 
of Lipschitz volumes~\cite{gromov}*{Section~4.5}, denoted  by $\norm{V:\vol}_+$ and 
$\norm{V:\vol}_-$, which also make sense for Riemannian manifolds~$V$ with 
infinite volume. For a complete Riemannian manifold~$V$ of finite 
volume he states the 
inequality~\cite{gromov}*{third set of examples in Section~4.5} 
\[\norm{V:\vol}_-\le\frac{\norm{V}}{\vol(V)}\le\norm{V:\vol}_+.\]

The hypothesis of being \emph{stable at
infinity}~\cite{gromov}*{proportionality theorem of Section~4.5} is
always satisfied for complete Riemannian manifolds of finite volume
(see the Example in \emph{loc.~cit.}), and thus the theorem in
\emph{loc.~cit.} says that
\[\norm{N:\vol}_-=\norm{N:\vol}_+=\frac{\norm{M}}{\vol(M)}. \]
Therefore, the inequality above yields the assertion. 
\end{proof}

Since the proof of Gromov's theorem is not very detailed, we  
refer the reader also to the following 
result~\cite{clara-roman}, which one might use alternatively 
in Strategy~\ref{strategy} below.  

\begin{theorem}[Proportionality principle, non-positive curvature]
\label{thm:proportionality principle}
Let $M$ and $N$ be complete, non-pos\-i\-tive\-ly curved 
Riemannian manifolds of finite volume 
whose universal covers are isometric. Then 
\[\frac{\lipvol{M}}{\vol(M)}=\frac{\lipvol{N}}{\vol(N)}.\]
\end{theorem}

Our strategy to compute the simplicial volume for a 
locally symmetric 
space $M$ of non-compact type and finite volume consists of the following steps: 

\begin{strategy}\label{strategy}\hfill
\begin{itemize}
\item Show that $\norm{M}=\lipvol{M}$ under certain conditions. 
\item By a theorem of Borel~\cite{borelcompact}, 
there is a compact locally symmetric 
space~$N$ of non-compact type 
with the same universal covering as $M$; furthermore, both
$M$ and~$N$ are 
non-positively curved 
(see~\cite{eberlein}*{Sections~2.1 and~2.2}).
\item Use the positivity result of Lafont and Schmidt~\cite{ben-jean} to 
conclude $\norm{N}>0$, or use other available 
specific computations of $\norm{N}$. 
\item Apply Theorem~\ref{thm:gromov proportionality principle} 
or~\ref{thm:proportionality principle} and the equality 
of the first step to get positivity or an explicit computation 
of $\norm{M}$ from the one of $\norm{N}$. Notice that both
proportionality principles are applicable in view of the second step. 
\end{itemize}
\end{strategy}

\subsection{Relating simplicial volume and Lipschitz simplicial volume}\label{subsec:general theorems}

The following general result helps to verify the 
first step in the strategy above for certain examples. 
It is proved in Section~\ref{sec:prooflipvol}. 

\begin{theorem}\label{thm:Lipschitz volume equals simplicial volume}
Let $(M,d)$ be an open, Riemannian manifold. Assume that 
there is a compact submanifold $W\subset M$  
with boundary $\partial W=\coprod_{i=1}^mW_i$ and connected 
components $W_i$ satisfying the following conditions: 
\begin{enumerate}
\item The complement~$M-W$ and the disjoint
  union~$\coprod_{i=1}^m W_i \times (0,\infty)$ of cylinders are
  homeomorphic. 
\item Each~$W_i$ has 
a finite cover $\overline{W}_i\rightarrow W_i$ that has a self-map 
$f_i:\overline{W}_i\rightarrow\overline{W}_i$ with~$\deg f_i \not\in \{-1,0,1\}$ 
with the 
following property: Let $f_i^{(k)}$ denote the $k$-fold composition
of $f_i$. There is $C>0$ such that for every $k\ge 1$ the map 
\[ F_i^{(k)}:\overline{W}_i\times [0,1]\rightarrow \overline{W}_i\times
[k,k+1],~(x,t)\mapsto\bigl(f_i^{(k)}(x),t+k\bigr)\]  
has Lipschitz constant at most~$C$ with respect to the metric
on~$\overline W_i \times (0,\infty)$ induced by~$d$. 
\item Each $W_i$ has amenable fundamental group. 
\end{enumerate}
Then 
\[ \norm{W,\partial W}=\norm{M}=\lipvol{M}. \]
\end{theorem}

\subsection{Simplicial volume of locally symmetric spaces of finite volume}
Next we present examples of locally symmetric spaces 
that satisfy the hypothesis of the previous theorem. 
In the sequel we focus on the case of higher $\bbR$-rank 
but we remark that it is not hard to verify the hypothesis 
for hyperbolic manifolds of finite volume or, 
more generally, locally symmetric spaces 
of $\bbR$-rank $1$. In particular, using 
Theorem~\ref{thm:gromov proportionality principle} 
or~\ref{thm:proportionality principle}, one obtains a new proof 
of Thurston's proportionality principle for 
hyperbolic manifolds~\cite{gromov}*{Section~0.4}. 

Let us first fix the following notation. 

\begin{setup}\label{setup:locally symmetric space}
Let $X$ be a connected symmetric space of non-compact type. 
Let $G=\isom(X)^\circ$ be the connected component of the identity in the 
isometry group of $X$. Then $G$ is a connected semi-simple Lie 
group with trivial center and no compact 
factors~\cite{eberlein}*{Proposition~2.1.1 on p.~69}. 
Let $\bfG$ be a connected, semi-simple, linear algebraic group 
with trivial center such that $G=\bfG(\bbR)$ (there 
always exists such $\bfG$). 
Assume that the~$\bbQ$-rank of~$\bfG(\bbQ)$ is~$1$, i.e., any  
maximal $\bbQ$-split torus of $\bfG(\bbQ)$ is one-dimensional. 
Let $\Gamma\subset\bfG(\bbQ)$ be a torsion-free, arithmetic lattice. 
By definition, one says that $\Gamma$ is a $\bbQ$-rank $1$ lattice. 
\end{setup}

Section~\ref{sec:review} gives a summary of the relevant notions 
about locally symmetric spaces. The following Theorem is proved 
in Section~\ref{sec:proofqrank}. 

\begin{theorem}\label{thm:q-rank 1 with amenable boundary satisfies assumption}
We retain the assumptions of Setup~\ref{setup:locally symmetric
  space}. Assume that $\Gamma$ in addition satisfies the 
following conditions: 
\begin{enumerate}
\item The lattice~$\Gamma$ is neat in the sense of
      Borel~\cite{oldborelbook}*{\S17}, and 
\item the connected components of the boundary of the Borel-Serre
      compactification of~$M=\Gamma\bs X$ have amenable fundamental
      groups.
\end{enumerate} 
Then $M$ satisfies the 
assumption of Theorem~\ref{thm:Lipschitz volume equals simplicial volume}. 
\end{theorem}

Applying the Strategy~\ref{strategy}, we obtain the following
positivity result:

\begin{corollary}\label{cor:positivity}
Under the hypotheses of the previous theorem we have 
$\norm{\Gamma\bs X}>0$. 
\end{corollary}

\subsection{Simplicial volume of Hilbert modular varieties}
Consider a totally real number field $F$ of degree $d$ over 
$\bbQ$, that is, $F$ admits no complex embedding. Let $\calo_F$ 
be its ring of integers. Let $\{\sigma_1,\ldots,\sigma_d\}$ be the 
set of all embeddings $\calo_F\rightarrow\bbR$. 

Then $\mathrm{SL}(2, \calo_F)$ is a $\bbQ$-rank $1$ 
lattice in $\mathrm{SL}(2,\bbR)^d$ via the embedding 
\[A\mapsto\bigl(\sigma_1(A),\ldots,\sigma_d(A)\bigr).\] 
In particular, $\mathrm{SL}(2, \calo_F)$ acts on the 
$d$-fold product $\bbH^2\times\dots\times\bbH^2$ of 
hyperbolic planes. 

If $\Gamma\subset \mathrm{SL}(2, \calo_F)$ is a torsion-free, finite
index subgroup then we call the quotient~$\Gamma\bs (\bbH^2)^d$ a
\emph{Hilbert modular variety}.  In case $F$ is real quadratic, i.e.,
$F=\bbQ(\sqrt{a})$ for a square-free positive integer $a$, we call
$\Gamma\bs \bbH^2\times\bbH^2$ a \emph{Hilbert modular surface}.

If $\Gamma$ is neat then $\Gamma\bs (\bbH^2)^d$ satisfies 
the assumption of Theorem~\ref{thm:q-rank 1 with amenable boundary
  satisfies assumption}~\cite{borelbook}*{III.2.7}. 
Further, every torsion-free lattice has a neat subgroup of 
finite index~\cite{oldborelbook}*{\S 17}. Thus, by
Corollary~\ref{cor:positivity} and the volume estimate  
in Section~\ref{subsec:minimal volume} we obtain: 

\begin{corollary}\label{cor:hilbert modular varieties}
Hilbert modular varieties have positive simplicial and positive  
minimal volume. 
\end{corollary}

To our knowledge, this is the first class of 
examples of non-compact 
locally symmetric spaces of $\bbR$-rank at least~$2$ for which 
positivity of simplicial or minimal volume is known. 

Bucher-Karlsson proved~\cite{michelle} that if the Riemannian
universal covering of a closed Riemannian manifold~$N$ is isometric
to~$\bbH^2\times\bbH^2$, then
$\norm{N}=\frac{3}{2\pi^2}\vol(N)$. Thus, applying
Strategy~\ref{strategy}, yields following explicit computation.

\begin{corollary}\label{cor:hilbert modular surface}
Let $\Sigma$ be a Hilbert modular surface. Then 
\[\norm{\Sigma}=\frac{3}{2\pi^2}\vol(\Sigma). \]
\end{corollary}

\begin{acknow}
The first author would like to thank the Graduiertenkolleg
``Analytische Topologie und Metageometrie'' at the WWU M\"unster for
its financial support.  The second named author would like to thank
Uri Bader for numerous discussions about Lie groups and symmetric
spaces, which helped a lot in the preparation of this paper. The second
named author acknowledges support of the German Science Foundation
(DFG), made through grant SA 1661/1-1.
\end{acknow}

\section{Proof of Theorem~\ref{thm:Lipschitz volume equals simplicial volume}}
\label{sec:prooflipvol}

\begin{proof}[Proof of Theorem~\ref{thm:Lipschitz volume equals
  simplicial volume}]
We first show that~$\norm{W, \partial W} \leq \norm{M} \leq \lipvol M$: By
definition, $\norm M \leq \lipvol M$. On the other hand, restricting a
locally finite fundamental cycle~$\ucov c$ of~$M$ to the submanifold~$W$, and
pushing the resulting finite chain down to~$W$ via the obvious
homotopy equivalence~$M \rightarrow W$ gives rise to a relative
fundamental cycle~$c$ of~$(W, \partial W)$; by construction, $\abs{c}_1 \leq
\abs{\ucov c}_1$. Therefore, $\norm{W,\partial W} \leq \norm{M}$.

It remains to prove the 
estimate~$\norm{W,\partial W} \geq \lipvol M$:
Let $c\in C_n(W)$ be a relative fundamental cycle of $(W,\partial W)$.
We now proceed in several steps: First, we use the given finite
coverings and self-maps on the boundary components to extend~$c$ to a
locally finite chain~$\ucov c$ on~$M$. In the second step, we show
that $\ucov c$ is a fundamental cycle of~$M$. By refining the
construction in the third step, we can build~$\ucov c$ in such a way
that is has small $\ell^1$-norm. In the last step, we show that
$\abs{\ucov c}_1$ is indeed small enough, by applying Gromov's
equivalence theorem. 

\medskip

\noindent\emph{Construction of a 
locally finite fundamental cycle~$\ucov c$ of~$M$ from~$c$.}
In the following, we write~$n = \dim M$.  
For $i\in\{1,\ldots, m\}$ let $p_i:\overline{W}_i\rightarrow W_i$
denote the given finite covering. Let $d_i\not\in\{0,-1,1\}$ be the
degree of the self-map~$f_i:\overline{W}_i\rightarrow\overline{W}_i$.
The chain
\[ z_i = (\partial c)\vert_{W_i}\in C_{n-1}(W_i). \]
is a fundamental cycle of~$W_i$ and $\partial c=\sum_iz_i$.  Let
$\overline z_i \in C_{n-1}(\overline W_i)$ be the image of~$z_i$ under the
transfer homomorphism~$C_\ast(W_i)\rightarrow C_\ast(\overline{W}_i)$
corresponding to~$p_i$. Then $\overline z_i$ is a cycle representing
$1/\abs{\deg p_i}$ times the fundamental class of~$\overline W_i$ and
satisfying $\abs{\overline z_i}_1\le \abs{z_i}_1$ as well as
$C_{n-1}(p_i)(\overline z_i)= z_i$.  Because $f_i$ has degree $d_i$,
we can find a chain $\overline b_i\in C_{n}(\overline W_i)$ such that
\begin{equation}\label{eq:boundary}
\partial \overline b_i=1/d_i\cdot C_{n-1}(f_i)(\overline z_i) - \overline z_i. 
\end{equation}
In the third step of the proof, we will specify a suitable choice
of~$\overline b_i$.

There is a chain~$\overline u_i\in C_n(\overline W_i\times [0,1])$ with
\[ \partial \overline u_i = j_1(\overline z_i)-j_0(\overline z_i)
   \text{ and }
   \abs{\overline u_i}_1\le n \cdot \abs{\overline z_i}_1,
\]
where $j_k : C_\ast(\overline W_i) \rightarrow C_\ast(\overline W_i
\times [0,\infty))$ denotes the map induced by the inclusion~$\overline
W_i \hookrightarrow \overline W_i \times \{k\} \hookrightarrow
\overline W_i \times [0,\infty)$. For example, such a~$\overline u_i$
can be found by triangulating the prism $\Delta^{n-1}\times [0,1]$
into $n$-simplices. Using the abbreviations
\begin{align*}
\overline w_i& = \overline u_i+j_1(\overline b_i)\in C_n\bigl(\overline{W}_i\times [0,1]\bigr),\\
\overline c_i& = \sum_{k\ge 0}\frac{1}{d_i^k}\cdot
     C_n\bigl(F_i^{(k)}\bigr)(\overline w_i)\in\Clf_n\bigl(\overline{W}_i\times [0,\infty)\bigr),
\end{align*}  
we define the (locally finite) chain
\begin{equation}\label{eq:locally finite chain}
\ucov c = c+\sum_{i=1}^m C_n(p_i)(\overline c_i)\in\Clf_n(M).
\end{equation}

Taking advantage of a smoothing
operator~\cite{loehmh}*{Proposition~5.3} from the singular chain
complex to the subcomplex generated by all smooth simplices, shows
that we can assume without loss of generality that all the simplices
occurring in the chains~$\overline w_i$ and~$c$ are smooth, and hence
Lipschitz; in particular, all simplices in~$\ucov c$ are Lipschitz.

\medskip

\noindent\emph{Properties of the chain $\ucov c$.}
The chain~$\ucov c$ is a cycle because
\begin{align*}
\partial \overline c_i
  &= \sum_{k\ge 0}
     \frac{1}{d_i^k}\cdot 
     C_{n-1}\bigl(F_i^{(k)}\bigr) (\partial \overline w_i)
     \\
  &= \sum_{k\ge 0}
     \frac{1}{d_i^k}\cdot 
     C_{n-1}\bigl(F_i^{(k)}\bigr)
            \bigl( j_1(\overline z_i)
                 - j_0(\overline z_i)
                 + j_1(\partial \overline b_i)
            \bigr)\\
  &= \sum_{k\ge 0}
     \frac{1}{d_i^k}\cdot 
     C_{n-1}\bigl(F_i^{(k)}\bigr)
            \Bigl( \frac{1}{d_i} \cdot 
                   j_1 \circ C_{n-1}(f_i)(\overline z_i)
                 - j_0(\overline z_i)
            \Bigr)\\
  &= \sum_{k\ge 0}
     \frac{1}{d_i^{k+1}}\cdot 
     j_{k+1} \circ C_{n-1}\bigl(f_i^{(k+1)}\bigr)(\overline z_i)
   - \sum_{k\ge 0}
     \frac{1}{d_i^k}\cdot 
     j_k \circ C_{n-1}\bigl(f_i^{(k)}\bigr)(\overline z_i)\\
  &= - \overline z_i. 
\end{align*}
For every~$x$ in the interior of $W$, the chains~$c$ and~$\ucov c$ give
rise to the same class in~$H_n(M, M - \{x\})$. Therefore, the
cycle~$\ucov c$ is a fundamental cycle of~$M$.
Moreover, $\lipschitz(c)<\infty$ follows directly from the 
uniform Lipschitz hypothesis for~$(F_i^{(k)})_{k \in\bbN}$. 

\medskip

\noindent\emph{Uniform boundary condition from amenability.} 
Because $\pi_1(W_i)$ is amenable, the finite index
subgroup~$\pi_1(\overline W_i)$ is also amenable. Thus, the chain
complex~$C_\ast(\overline W_i)$ satisfies the uniform boundary
condition of Matsumoto and Morita~\cite{morita}*{Theorem~2.8} in
degree~$n - 1$ (in fact, in every degree), i.e., there is a
constant~$K>0$ such that for any null-homologous cycle~$\overline z\in
C_{n-1}(\overline W_i)$ there is a chain $\overline b\in
C_{n}(\overline W_i)$ with $\partial \overline b=\overline z$ and
$\abs{\overline b}_1\le K\cdot\abs{\overline z}_1$. By taking the
maximum we may assume that we have the same constant~$K$ for
all~$i\in\{1,\ldots, m\}$.  Therefore, we can choose the~$\overline
b_i$ above in such a way that 
\begin{equation*}
\abs{\overline b_i}_1
           \le K\cdot\bigl| 1/d_i\cdot C_{n-1}(f_i)(\overline z_i)
                          - \overline z_i
                     \bigr|_1 
           \le 2K\cdot \abs{\overline z_i}_1,
\end{equation*}
and again -- by smoothing~$\overline b_i$ -- we may assume that all
simplices occurring in~$\overline b_i$ are smooth. 
This yields 
\begin{align*}
\abs{\ucov c}_1
&\le\abs{c}_1+\sum_{i=1}^m\abs{\overline c_i}_1\\
&\le\abs{c}_1+\sum_{i=1}^m\sum_{k\ge 0}
  \frac{1}{d_i^k}\cdot\bigl(\abs{\overline u_i}_1+\abs{j_1(\overline b_i)}_1\bigr)\\
&\le\abs{c}_1+\sum_{i=1}^m\sum_{k\ge 0}
  \frac{2K+n}{d_i^k}\cdot\abs{\overline z_i}_1\\
&\le\abs{c}_1+2m(2K+n)n\cdot \abs{\partial c}_1.
\end{align*}

\medskip

\noindent\emph{Applying Gromov's equivalence theorem.}
From the discussion above we obtain that 
\[\lipvol{M}\le\inf\bigl\{\abs{c}_1+2m(2K+n)n\cdot\abs{\partial
  c}_1;~\text{$c\in C_n(W)$ relative fundamental cycle}
\bigr\}.\]
Then the amenability of each $\pi_1(W_i)$ and 
Gromov's equivalence theorem~\cite{gromov}*{p.~57} imply that the right hand 
side is equal to $\norm{W,\partial W}$. 
\end{proof}

\section{Review of cusp decomposition in $\bbQ$-rank
  $1$}\label{sec:review}  

\noindent%
We refer to the Setup~\ref{setup:locally symmetric space}. 
Our background reference is the book of Borel and
Ji~\cite{borelbook}. As in this book, we stick to the following notation:

\subsection*{Convention} \textit{An algebraic group defined over 
$\bbQ$ or $\bbR$ is denoted by a bold face letter and its 
group of real points by the corresponding Roman capital. 
Parabolic subgroups are always assumed to be proper and 
usually denoted by $\bfP$.}

\subsection{The Borel-Serre compactification -- simply connected
  case}\label{subsec:horospherical coordinates} 
In the \mbox{$\bbQ$-rank} one case, 
the 
\emph{Borel-Serre compactification $\overline{X}$} of the symmetric space~$X$ (the one 
of the locally symmetric space~$\Gamma\bs X$ will then be
$\Gamma\bs\overline{X}$, cf.~Section~\ref{subsec:the quotient description}) 
is of the form 
\[\overline{X}=X\cup\coprod_{\bfP}e(\bfP),\]
where $\bfP$ runs over all rational parabolic subgroups of $\bfG$, 
and the boundary component $e(\bfP)$ is a principal bundle over a symmetric 
space $X_\bfP$ with the unipotent radical $N_\bfP$ of $P$ as fiber. 
If one fixes a basepoint $x_0\in X$ and thus an identification 
$X=G/K$ (with $K=G_{x_0}$ maximal compact), then one can describe $e(\bfP)$ as 
follows~\cite{borelbook}*{Section~III.9}: Let $P=N_\bfP A_\bfP M_\bfP$ 
be the \emph{rational Langlands decomposition} of $P$ with respect to $x_0$: 
Here $A_{\bfP}$ is a stable lift of the identity component of the real locus 
of the maximal $\bbQ$-split torus in the Levi-quotient $\bfN_\bfP\bs\bfP$, 
and $M_\bfP$ is a stable lift of the real locus of the complement of this torus in 
$\bfN_\bfP\bs\bfP$. Write $K_\bfP=K\cap M_\bfP$ and
$\Gamma_P=\Gamma\cap P$ 
and $\Gamma_{N_\bfP}=\Gamma\cap N_\bfP$. 
Then $X_\bfP=M_\bfP/K_\bfP$ and $e(\bfP)=N_\bfP \times X_\bfP$. 
The map~\cite{borelbook}*{Section~III.1, p.~272-275} 
\begin{equation*}\label{eq:horospherical coordinates}
\mu_{x_0}: N_\bfP \times A_\bfP\times X_\bfP\xrightarrow{\cong}
X,~(n,a,mK)\mapsto nam\cdot x_0
\end{equation*}
is a diffeomorphism. If $\mu_{x_0}(n,a,z)=x$ then one calls $(n,a,z)$ 
the \emph{horospherical coordinates} of $x\in X$. 

\subsection{The rational root space decomposition}\label{subsec:root
  spaces} 
Let $\frakg$ be the Lie algebra of~$G$.  Let $\fraka_\bfP$ and
$\frakn_\bfP$ be the Lie algebra of $A_\bfP$ and $N_\bfP$,
respectively. Associated with~$(\frakg, \fraka_\bfP)$ is the system
$\Phi(\frakg, \fraka_\bfP)$ of $\bbQ$-roots. For
$\alpha\in\Phi(\frakg, \fraka_\bfP)$ we write
\[\frakg_\alpha=\{Z\in\frakg;\ \forall_{H\in\fraka_\bfP}\ad(H)(Z)=\alpha(H)Z
\}.\]
Furthermore, if $\Phi^+(\frakg,\fraka_\bfP)$ denotes the subset
of positive roots, then we have
\begin{equation*}\label{eq:root decomposition}
\frakn_\bfP=\bigoplus_{\alpha\in\Phi^+(\frakg,\fraka_{\bfP})}\frakg_\alpha.
\end{equation*}
Because $\bfG(\bbQ)$ has 
$\bbQ$-rank $1$ there is only one positive simple root
$\alpha_0\in\Phi^+(\frakg,\fraka_\bfP)$,  
which defines a group isomorphism  
\begin{equation*}\label{eq:parametrizing the split component}
A_\bfP\xrightarrow{\cong}\bbR^\ast_{>0},~a\mapsto
a^{\alpha_0}\defq\mathrm{exp}\circ \alpha_0\bigl(\log(a)\bigr).   
\end{equation*}
We remark that 
the root decomposition above gives a grading of $\frakn_\bfP$, and 
the adjoint action of $a\in A_\bfP$ on $\frakn_\bfP$ is given by 
multiplication with $\lambda^k$ on $\frakg_{k\alpha_0}$, where 
$\lambda=a^{\alpha_0}$. 

\subsection{The ends, metrically}\label{subsec:neigborhood and metric}
For any $t\ge 0$ let $A_{\bfP,t}=\{a\in A_\bfP;~\alpha_0(\log a)\ge t\}$. 
Then 
\[V_\bfP(t)=\mu_{x_0}\bigl (N_\bfP\times A_{\bfP,t}\times X_\bfP\bigr)\]
together with $e(\bfP)$ 
defines a neighborhood of $e(\bfP)$ for any $t\ge 0$. 
Next we describe 
the metric on $V_\bfP(t)$ in horospherical
coordinates~\cite{borelbook}*{Lemma~III.20.7 on p.~402}:  
The tangent spaces 
of the three submanifolds $N_\bfP \times\{a\}\times\{z\}$ and 
$\{n\}\times A_{\bfP,t}\times\{z\}$ and 
$\{n\}\times\{a\}\times X_\bfP$ 
at every point $(n,a,z)$ are orthogonal. 
Let $dx^2$, $da^2$, and $dz^2$ be the invariant 
metrics on $X$, $A_\bfP$, and $X_\bfP$, respectively, induced 
from the Killing form. Then we have 
\begin{equation*}\label{eq:cusp metric}
dx^2=
\sum_{a\in\Phi^+(\frakg,\fraka_\bfP)}e^{-2\alpha(\log(a))}h^\alpha(z)\oplus 
da^2\oplus dz^2,
\end{equation*}
where $h^\alpha(z)$ is some metric on
$\frakg_\alpha$ that smoothly depends on $z$ but is independent of~$a$.

\subsection{The Borel-Serre compactification -- locally symmetric
  case}\label{subsec:the quotient description} 
Let $S$ denote a set of representatives of 
$\Gamma$-conjugacy classes of rational parabolic subgroups. 
There is a free $\Gamma$-action on $\overline{X}$ extending that of 
$X$, and the quotient~$\Gamma\bs\overline{X}$ is the 
\emph{Borel-Serre compactification} of the locally symmetric space 
$\Gamma\bs X$. The $\Gamma$-stabilizer of $e(\bfP)$, and also 
of $V_\bfP(t)$ for~$t\gg1$, is $\Gamma_P$, and two
points in $V_\bfP(t)$ lie in the same $\Gamma$-orbit if and only if 
they lie in the same $\Gamma_P$-orbit. 
It turns out that
\begin{equation*}
\Gamma\bs\overline{X}=\Gamma\bs X\cup\coprod_{\bfP\in S}\Gamma_P\bs e(\bfP).
\end{equation*}
Moreover, for $t\gg 1$ we have (in horospherical coordinates with
respect to a basepoint) 
\begin{equation*}
\Gamma_P\bs V_\bfP(t)=\Gamma_P\bs\bigl (N_\bfP\times A_{\bfP,t}\times
X_\bfP\bigr)=\Gamma_P\bs\bigl (N_\bfP\times X_\bfP\bigr)\times A_{\bfP,t}.
\end{equation*}
There is $t_0>0$ such that for all $t>t_0$ the subsets 
$(\Gamma_P\bs V_\bfP(t))_{\bfP\in S}$, of $\Gamma\bs X$ are disjoint and 
the complement $\Gamma\bs X-\bigcup_{\bfP\in S}\Gamma_P\bs V_\bfP(t)$ is 
a compact submanifold with boundary. 
For $t>t_0$ we say that $\Gamma_P\bs V_\bfP(t)$ is a
\emph{cusp}. The union of all cusps together with the boundary 
of $\Gamma\bs\overline{X}$ is a collar neighborhood 
of the same boundary~\cite{lizhen}*{Proposition~2.3 and~2.4}. 

\subsection{The boundary of the Borel-Serre compactification and neatness}\label{subsec:fiber bundle}
Denote the image of $\Gamma_P$ under the projection 
$P=N_\bfP A_{\bfP}M_{\bfP}\rightarrow M_{\bfP}$ by
$\Gamma_{M_\bfP}$. We assume that 
$\Gamma$ is \emph{neat} in the sense of 
Borel~\cite{oldborelbook}*{\S 17}. By definition, $\Gamma$ is 
neat if the multiplicative subgroup of $\bbC^\times$ generated 
by the eigenvalues of $f(\gamma)\in\mathrm{GL}(n,\bbC)$ is torsion-free for any 
$\gamma\in\Gamma$ and any $\bbQ$-embedding 
$f:G\rightarrow\mathrm{GL}(n,\bbC)$. 
Every lattice has a neat subgroup of 
finite index, and a neat lattice is also
torsion-free~\cite{oldborelbook}*{\S 17}.  
Then $\Gamma_{M_\bfP}\subset M_{\bfP}$ is a torsion-free, uniform 
lattice in~$M_\bfP$, and one obtains a short exact sequence 
\[0\rightarrow\Gamma_{N_\bfP }\rightarrow\Gamma_P
\rightarrow\Gamma_{M_\bfP}\rightarrow 0.\]
Moreover, $\Gamma_P\bs (N_\bfP\times X_\bfP)$ is a fiber bundle over the locally
symmetric space $B_\bfP=\Gamma_{M_\bfP}\bs X_\bfP$ whose fiber is the
nil-manifold $\Gamma_{N_\bfP}\bs N_\bfP$. The assumption that 
$\Gamma$ is neat is needed 
to ensure that $B_\bfP$ is a manifold. 
The homeomorphism 
$\alpha_0\circ\log: A_{\bfP,t}\rightarrow\bbR_{>t}$ identifies the
metric on $A_{\bfP,t}$ with the standard Euclidean metric on~$\bbR$. 
Every cusp of $\Gamma\bs X$ is topologically a cylinder
$\Gamma_P\bs(N_\bfP\times X_\bfP)\times\bbR_{>t}$. From the discussion 
in Section~\ref{subsec:neigborhood and metric} we see that in the locally
symmetric metric the nil-manifold fiber in $B_\bfP$ shrinks
exponentially while the base stays fixed when the
$A_{\bfP,t}$-component 
goes to infinity. 

\section{Proof of  Theorem~\ref{thm:q-rank 1 with amenable boundary
  satisfies assumption}}\label{sec:proofqrank}

\begin{proof}[Proof of Theorem~\ref{thm:q-rank 1 with amenable
    boundary satisfies assumption}] 
We know from the discussions in Section~\ref{sec:review}
(especially~\ref{subsec:the quotient description}), to the notation of
which we freely refer, that the locally symmetric space~$M=\Gamma\bs
X$ has finitely many cusps~$(\Gamma_P\bs V_\bfP(t))_{\bfP\in S}$,
where $S$ is a set of representatives of \mbox{$\Gamma$-con}\-jugacy
classes of rational parabolic subgroups and $t>0$ is sufficiently 
large. The complement $W$ of $M$
minus the union of cusps is diffeomorphic to the Borel-Serre
compactification of $M$. It remains to verify hypothesis b) of
Theorem~\ref{thm:Lipschitz volume equals simplicial volume} for every
cusp $\Gamma_P\bs V_\bfP(t)$. We consider the rational Langlands
decomposition $P=N_\bfP A_\bfP M_\bfP$ with respect to a basepoint
$x_0\in X$ such that $A_\bfP$ and $M_\bfP$ are defined over
$\bbQ$. This is always possible~\cite{boreljipaper}*{Proposition~2.2}.
Then $\Lambda_\bfP\defq\Gamma_{N_\bfP}(\Gamma\cap M_\bfP)$ is a finite
index subgroup of~$\Gamma_P$~\cite{boreljipaper}*{Remark~2.7}.  In
horospherical coordinates, the cusp is given by
\[\Gamma_P\bs V_\bfP(t)=
\Gamma_P\bs\bigl( N_\bfP\times X_\bfP\bigr)\times A_{\bfP,t}.\]
Moreover, $\overline W_{\bfP}= \Lambda_\bfP\bs (N_\bfP\times X_\bfP)$ is 
a finite cover 
of $W_{\bfP} = \Gamma_P\bs\bigl( N_\bfP\times X_\bfP\bigr)$. 

Note for the following 
that $N_\bfP\times A_\bfP\times X_\bfP$ inherits a $G$-action from~$X$
via~$\mu_{x_0}$. 
In coordinates, multiplication by~$a_0\in A_\bfP$ is given 
by $a_0(n, a, x)=(c_{a_0}(n), a_0a, x)$ where 
$c_{a_0}(n)=a_0na_0^{-1}$~\cite{boreljipaper}*{Section~2}. 
Remember also that the (positive) root system
$\Phi^+(\frakg,\fraka_\bfP)$ in Section~\ref{subsec:root spaces} 
consists of positive, integral multiples of the single, positive
simple root $\alpha_0$. 

The next lemma provides the non-trivial self-maps (even self-coverings)
assumed in hypothesis b). \noqed
\end{proof}
\begin{lemma}
There is a non-trivial self-covering 
$f_\bfP$ of $\Lambda_\bfP\bs (N_\bfP\times X_\bfP)$ 
such that the following square commutes:
\[\xymatrix{
\Lambda_\bfP\bs (N_\bfP\times X_\bfP)\ar[r]\ar[d]_{f_\bfP}
& (\Gamma\cap M_\bfP)\bs X_\bfP\ar[d]^\id\\
\Lambda_\bfP\bs (N_\bfP\times X_\bfP)\ar[r]& (\Gamma\cap M_\bfP)\bs X_\bfP
}\]
Moreover, the lift of $f_\bfP$ to the universal cover is given by 
$(n,x)\mapsto (a_0na_0^{-1}, x)$ for some $a_0\in A_\bfP$ with 
$\lambda=a_0^{\alpha_0}>1$.
\end{lemma}

\begin{proof}[Proof of Lemma]
Because $A_\bfP$ commutes with $\Gamma\cap M_\bfP$, 
the map~$c_{a_0}\times\id_{X_\bfP}$ for some~$a_0\in A_\bfP$ induces 
a non-trivial self-covering on the $\Lambda_\bfP$-quotient provided 
$c_{a_0}(\Gamma_{N_\bfP})$ is a proper, finite index subgroup of 
$\Gamma_{N_\bfP}$. The latter condition can be checked on the 
level of Lie algebras. Because $N_\bfP$ and $A_\bfP$ are defined 
over~$\bbQ$, the Lie algebra~$\frakn_\bfP$ has a preferred rational 
structure and the decomposition of~$\frakn_\bfP$ 
in Section~\ref{subsec:root spaces} is defined over~$\bbQ$. Recall that 
the adjoint action of~$a_0$ on~$\frakn_\bfP$ is given by the 
homomorphism~$\phi_\lambda:\frakn_\bfP\rightarrow\frakn_\bfP$ where 
$\lambda=a_0^{\alpha_0}$ and 
$\phi_\lambda(n)=\lambda^in$ for~$n\in\frakg_{i\alpha_0}$. 
The group~$\Gamma_{N_\bfP}$ corresponds to a $\bbZ$-lattice 
in $\frakn_\bfP$ by the exponential map. Dekimpe and Lee show 
 that there is an integer~$\lambda>1$ such that 
$\phi_\lambda$ preserves this lattice~\cite{dekimpe}*{Lemma~5.2}. 
The corresponding~$a_0\in
A_\bfP$ with $a_0^{\alpha_0}=\lambda$ is the desired element. 
\end{proof}

\begin{proof}[Continuation of the proof of 
Theorem~\ref{thm:q-rank 1 with amenable boundary satisfies assumption}] 
Let $f_\bfP$ and $a_0\in A_{\bfP}$ and $\lambda>1$ be as in the previous lemma. 
Recall that $\Phi^+(\frakg,\fraka_\bfP)=\{\alpha_0,2 \alpha_0,\ldots\}$, where 
$\alpha_0$ is the single positive simple root.  
Next we show that left multiplication by~$a_0$ is Lipschitz with Lipschitz 
constant smaller than~$1$:  
First note that left multiplication by 
$a_0$ on the $A_{\bfP,t}$- and $X_\bfP$-component 
of $V_\bfP(t)$ is an isometry or the identity, respectively. 
On the $N_\bfP$-component it is just the conjugation $c_{a_0}$. 
Let $\norm{V}_{(n,a,z)}$ denote the norm of a tangent vector $V$ of $V_\bfP(t)$ 
at the point $(n,a,z)$. 
For $m\in\bbN$ and $N\in\frakg_{m\alpha_0}$ we have 
according to Section~\ref{subsec:neigborhood and metric} 
\[\norm{dc_{a_0}(N)}_{(c_{a_0}(n),aa_0,z)}
   =\norm{\lambda^m N}_{(c_{a_0}(n),aa_0,z)}
   =\lambda^{-2m}\lambda^m\norm{N}_{(n,a,z)}<\norm{N}_{(n,a,z)}.\]
In particular, for any $k\in\bbN$, left multiplication by 
$a_0^k$ has a Lipschitz constant smaller than~$1$. 
Let $t_0=\alpha_0(\log a_0)>0$. If we identify $A_{\bfP, t}$
with~$\bbR_{>t}$ by~$\alpha_0\circ\log$, then the map on~$\Lambda_P\bs
V_\bfP(t)$ induced by left multiplication by~$a_0$ is just
\begin{equation*}
\Lambda_P\bs\bigl(N_\bfP\times X_\bfP\bigr)\times
\bbR_{>t}\ni ((n,z),t')\mapsto
(f_\bfP(n,z), t_0+t'). 
\end{equation*}
Thus, after a suitable reparametrization of $\bbR_{>t}$,  
multiplication by~$a_0^k$ is the desired map $F_i^{(k)}$ 
($i$ corresponds to $\bfP$) in hypothesis b) of 
Theorem~\ref{thm:Lipschitz volume equals simplicial volume}, 
which finishes the proof. 
\end{proof}


\begin{bibdiv}
\begin{biblist}

\bib{bcg}{article}{
   author={Besson, G\'erard},
   author={Courtois, Gilles},
   author={Gallot, Sylvestre},
   title={Entropies et rigidit\'es des espaces localement sym\'etriques de
   courbure strictement n\'egative},
   language={French},
   journal={Geom. Funct. Anal.},
   volume={5},
   date={1995},
   number={5},
   pages={731--799},
   issn={1016-443X},
}


\bib{juan}{article}{
   author={Boland, Jeffrey},
   author={Connell, Chris},
   author={Souto, Juan},
   title={Volume rigidity for finite volume manifolds},
   journal={Amer. J. Math.},
   volume={127},
   date={2005},
   number={3},
   pages={535--550},
   issn={0002-9327},
}

\bib{borelcompact}{article}{
   author={Borel, Armand},
   title={Compact Clifford-Klein forms of symmetric spaces},
   journal={Topology},
   volume={2},
   date={1963},
   pages={111--122},
   issn={0040-9383},
}

\bib{oldborelbook}{book}{
   author={Borel, Armand},
   title={Introduction aux groupes arithm\'etiques},
   language={French},
   series={Publications de l'Institut de Ma\-th\'ematique de l'Universit\'e de
   Strasbourg, XV. Actualit\'es Scientifiques et Industrielles, No. 1341},
   publisher={Hermann},
   place={Paris},
   date={1969},
   pages={125},
}
	

\bib{boreljipaper}{article}{
   author={Borel, Armand},
   author={Ji, Lizhen},
   title={Compactifications of locally symmetric spaces},
   journal={J. Differential Geom.},
   volume={73},
   date={2006},
   number={2},
   pages={263--317},
   issn={0022-040X},
}
		
\bib{borelbook}{book}{
   author={Borel, Armand},
   author={Ji, Lizhen},
   title={Compactifications of symmetric and locally symmetric spaces},
   series={Mathematics: Theory \& Applications},
   publisher={Birkh\"auser Boston Inc.},
   place={Boston, MA},
   date={2006},
   pages={xvi+479},
}

\bib{michelle}{article}{
   author={Bucher-Karlsson, Michelle },
   title={The simplicial volume of closed manifolds covered by
     $\bbH^2\times\bbH^2$},
   eprint={arXiv:math.DG/0703587},
   date={2007},
}

\bib{chris+benson}{article}{
   author={Connell, Christopher},
   author={Farb, Benson},
   title={The degree theorem in higher rank},
   journal={J. Differential Geom.},
   volume={65},
   date={2003},
   number={1},
   pages={19--59},
   issn={0022-040X},
}

\bib{dekimpe}{article}{
   author={Dekimpe, Karel},
   author={Lee, Kyung Bai},
   title={Expanding maps on infra-nilmanifolds of homogeneous type},
   journal={Trans. Amer. Math. Soc.},
   volume={355},
   date={2003},
   number={3},
   pages={1067--1077 (electronic)},
   issn={0002-9947},
}
		
\bib{eberlein}{book}{
   author={Eberlein, Patrick B.},
   title={Geometry of nonpositively curved manifolds},
   series={Chicago Lectures in Mathematics},
   publisher={University of Chicago Press},
   place={Chicago, IL},
   date={1996},
   pages={vii+449},
}
			
\bib{gromov}{article}{
   author={Gromov, Michael},
   title={Volume and bounded cohomology},
   journal={Inst. Hautes \'Etudes Sci. Publ. Math.},
   number={56},
   date={1982},
   pages={5--99},
   issn={0073-8301},
}

\bib{lizhen}{article}{
   author={Ji, Lizhen},
   author={Zworski, Maciej},
   title={Scattering matrices and scattering geodesics of locally symmetric
   spaces},
   language={English, with English and French summaries},
   journal={Ann. Sci. \smash{\'Ecole} Norm. Sup. (4)},
   volume={34},
   date={2001},
   number={3},
   pages={441--469},
   issn={0012-9593},
}

\bib{ben-jean}{article}{
   author={Lafont, Jean-Fran{\c{c}}ois},
   author={Schmidt, Benjamin},
   title={Simplicial volume of closed locally symmetric spaces of
   non-compact type},
   journal={Acta Math.},
   volume={197},
   date={2006},
   number={1},
   pages={129--143},
   issn={0001-5962},
}

\bib{loehmh}{article}{
   author={L\"oh, Clara},
   title={Measure homology and singular homology are isometrically
   isomorphic},
   journal={Math.~Z.},
   volume={253},
   date={2006},
   number={1},
   pages={197--218},
}

\bib{clara-roman}{article}{
   author={L\"oh, Clara},
   author={Sauer, Roman},
   title={Degree theorems and Lipschitz simplicial volume for
     non-positively curved manifolds of finite volume},
   date={2007},
   status={in preparation},
}

\bib{morita}{article}{
   author={Matsumoto, Shigenori},
   author={Morita, Shigeyuki},
   title={Bounded cohomology of certain groups of homeomorphisms},
   journal={Proc. Amer. Math. Soc.},
   volume={94},
   date={1985},
   number={3},
   pages={539--544},
   issn={0002-9939},
}

\bib{storm}{article}{
   author={Storm, Peter A.},
   title={The minimal entropy conjecture for nonuniform rank one lattices},
   journal={Geom. Funct. Anal.},
   volume={16},
   date={2006},
   number={4},
   pages={959--980},
   issn={1016-443X},
}
		



\end{biblist}
\end{bibdiv}
\end{document}